\documentclass[12pt]{amsart}

\usepackage{amssymb,latexsym}
\usepackage{amsthm,amsfonts}
\usepackage[centertags]{amsmath}
\usepackage[a4paper,margin=22mm]{geometry}

\newcommand{\N}{{\mathbb{N}}}

\newcommand{\Z}{{\mathbb{Z}}}

\newcommand{\F}{{\mathbb{F}}}

\newcommand{\x}{{\bf x}}
\newcommand{\e}{{\bf e}}
\newcommand{\y}{{\bf y}}

\def \se {sequence }

\newtheorem{theorem}{Theorem}
\newtheorem{proposition}{Proposition}
\newtheorem{corollary}{Corollary}
\newtheorem{lemma}{Lemma}

\newtheorem{definition}{Definition}

\newtheorem{remark}{Remark} 
\newtheorem{example}{Example}

\newcommand{\bsa}{\boldsymbol{a}}
\newcommand{\bsA}{\boldsymbol{A}}
\newcommand{\bsb}{\boldsymbol{b}}
\newcommand{\bsc}{\boldsymbol{c}}
\newcommand{\bsd}{\boldsymbol{d}}

\newcommand{\Field}{\mathbb{F}}
\newcommand{\NN}{\mathbb{N}}
\newcommand{\cN}{\mathcal{N}}

\newcommand{\cC}{\mathcal{C}}

\newcommand{\cL}{\mathcal{L}}
\newcommand{\bszero}{\boldsymbol{0}}
\newcommand{\bsone}{\boldsymbol{1}}
\newcommand{\To}{\rightarrow}
\allowdisplaybreaks

\begin{document}
\title[Propagation rules]{Propagation rules for $(u,m,\e,s)$-nets and $(u,\e,s)$-sequences}

\author{PETER KRITZER}
\address{Peter Kritzer, 
Department of Financial Mathematics, 
Johannes Kepler University Linz, 
Altenbergerstr. 69, 
A-4040 Linz, 
AUSTRIA}
\email{peter.kritzer@jku.at (P. Kritzer)}

\author{HARALD NIEDERREITER}
\address{
Harald Niederreiter,
Johann Radon Institute for Computational and Applied Mathematics,
Austrian Academy of Sciences,
Altenbergerstr. 69,
A-4040 Linz,
AUSTRIA,
and
Department of Mathematics, University of Salzburg, Hellbrunnerstr. 34,
A-5020 Salzburg, AUSTRIA
}
\email{ghnied@gmail.com (H. Niederreiter)}

\thanks{P. Kritzer gratefully acknowledges the support of the Austrian Science Fund (FWF), Project P23389-N18}

\date{\today} 



\begin{abstract}
The classes of $(u,m,\e,s)$-nets and $(u,\e,s)$-sequences were recently introduced by Tezuka, and in a slightly more restrictive 
form by Hofer and Niederreiter. We study propagation rules for these point sets, which state how one can obtain $(u,m,\e,s)$-nets 
and $(u,\e,s)$-sequences with new parameter configurations from existing ones. In this way, 
we show generalizations and extensions of several well-known construction 
methods that have previously been shown for $(t,m,s)$-nets and $(t,s)$-sequences. We also develop a duality theory for digital
$(u,m,\e,s)$-nets and present a new construction of such nets based on global function fields.  
\end{abstract}

\maketitle

\section{Introduction}\label{secintro}

Finite or infinite sequences of points with good equidistribution properties are frequently studied in number-theoretic questions, 
and they also play an important role as the node sets of quadrature rules in numerical integration (see, e.g., \cite{DP} and \cite{N92}). 
When studying the question of how to evenly distribute (a large number of) points in a certain domain, one frequently restricts oneself 
to considering the $s$-dimensional unit cube $[0,1]^s$. 

If we would like to distribute 
a large number of points in the unit cube, a very popular and powerful method is to use $(t,m,s)$-nets 
(for the case of finitely many points) and $(t,s)$-sequences (for the case of 
infinitely many points), or, more generally, the recently introduced $(u,m,\e,s)$-nets and $(u,\e,s)$-sequences. 
The former classes of point sets and sequences were introduced 
in their nowadays most common form by Niederreiter in~\cite{N87}, see also~\cite{N92} for detailed information, while the latter 
were introduced by Tezuka in~\cite{T13} and studied in a slightly 
modified form by Hofer and Niederreiter~\cite{HN13} and Niederreiter and Yeo~\cite{NY}. The underlying idea of these nets and sequences is to guarantee fair 
distribution of the points for certain subintervals 
of the half-open unit cube $[0,1)^s$. To be more precise, let $s\ge 1$ be a given dimension and let $b\ge 2$ be an integer 
(which we usually refer to as the base 
in the following). An interval $J\subseteq [0,1)^s$ is 
called an \emph{elementary interval in base} $b$ if it is of the form 
$$J=\prod_{i=1}^s \left[\frac{a_i}{b^{d_i}},\frac{a_i +1}{b^{d_i}}\right),$$
with integers $d_i\ge 0$ and $0\le a_i < b^{d_i}$ for $1\le i\le s$. These intervals play a crucial role in the subsequent 
definition of a $(u,m,\e,s)$-net, which we state below. 
Here and in the following, we denote by $\NN$ the set of positive integers and by $\lambda_s$ the $s$-dimensional Lebesgue measure.

\begin{definition}\label{defumes} {\rm
 Let $b\ge 2$, $s\ge 1$, and $0\le u\le m$ be integers and let $\e=(e_1,\ldots,e_s)\in\NN^s$. A point set $\mathcal{P}$ of $b^m$ 
points in $[0,1)^s$ is called a $(u,m,\e,s)$-\emph{net in base} $b$ if every elementary interval 
$J\subseteq [0,1)^s$ in base $b$ of volume $\lambda_s (J)\ge b^{u-m}$ and of the form 
$$J=\prod_{i=1}^s \left[\frac{a_i}{b^{d_i}},\frac{a_i +1}{b^{d_i}}\right),$$
with integers $d_i\ge 0$, $0\le a_i < b^{d_i}$, and $e_i |d_i$ for $1\le i\le s$, contains exactly $b^m \lambda_s (J)$ points of $\mathcal{P}$.}
\end{definition}

Note that the points of a $(u,m,\e,s)$-net are particularly evenly distributed if $u$ is small. On the other hand, also the choice of $e_1,\ldots,e_s$ 
plays an important role, as larger values of the $e_i$ in general 
mean less restrictions on the distribution of the points in the unit cube.

Definition~\ref{defumes} is the definition of a $(u,m,\e,s)$-net in base $b$ in the sense of~\cite{HN13}. Previously, Tezuka~\cite{T13} introduced 
a slightly more general definition, where the conditions 
on the number of points in the elementary intervals only need to hold for those elementary intervals $J$ in base $b$ with $\lambda_s (J) = b^{u-m}$. The 
narrower definition used in~\cite{HN13} guarantees, as stated in that 
paper, that any $(u,m,\e,s)$-net in base $b$ is also a $(v,m,\e,s)$-net in base $b$ for any integer $v$ with $u\le v\le m$. The latter property 
is a very useful property when working with such point sets (see also \cite{HN13} for further details); 
hence, whenever we speak of a $(u,m,\e,s)$-net here, we mean a $(u,m,\e,s)$-net in the narrower sense of Definition~\ref{defumes}. 

The definition of a $(u,\e,s)$-sequence is based on $(u,m,\e,s)$-nets and is given next. As usual, we write
$[\x]_{b,m}$ for the coordinatewise $m$-digit truncation in base $b$ of $\x \in [0,1]^s$.

\begin{definition}\label{defues} {\rm
Let $b\ge 2$, $s\ge 1$, and $u\ge 0$ be integers and let $\e\in\NN^s$. A sequence $\x_0,\x_1,\ldots$ of points in $[0,1]^s$ is a 
$(u,\e,s)$-\emph{sequence in base} $b$ if for all integers $k\ge 0$ and $m>u$ the 
points $[\x_n]_{b,m}$ with $kb^m \le n < (k+1)b^m$ form a $(u,m,\e,s)$-net in base $b$. } 
\end{definition}

Again, the points of a $(u,\e,s)$-sequence are evenly distributed if $u$ is small, but also in this case the choice of 
$\e$ has influence on the conditions on how the points are spread over the elementary intervals 
in the unit cube.

If we choose $\e=(1,\ldots,1) \in \N^s$ in Definitions~\ref{defumes} and~\ref{defues}, then the definitions coincide with those of a classical
$(t,m,s)$-net or a classical $(t,s)$-sequence with $t=u$, respectively, 
as they were introduced in~\cite{N87}. The reasons why the more general 
$(u,m,\e,s)$-nets and $(u,\e,s)$-sequences were introduced are twofold. On the one hand, as pointed out by Tezuka in~\cite{T13}, 
one can derive better bounds on the discrepancy for at least some examples of $(t,m,s)$-nets 
and $(t,s)$-sequences by viewing them as special cases of $(u,m,\e,s)$-nets and $(u,\e,s)$-sequences (see also the recent paper~\cite{FL13} for related results). 
Furthermore, by viewing, e.g., $(t,s)$-sequences as $(u,\e,s)$-sequences, it is possible to deal with special types of $(t,s)$-sequences in a very natural way. 
For example, as pointed out in~\cite{T13} 
(see also~\cite{HN13}), a generalized Niederreiter sequence over the finite field 
$\Field_q$ ($q$ a prime power) is a $(u,\e,s)$-sequence in base $q$ with $u=0$, where $\e=(e_1,\ldots,e_s)\in\NN^s$ is such that  
$e_i$, $1\le i\le s$, exactly corresponds to the degree of the $i$th base polynomial over $\Field_q$ used 
in the construction of the sequence. The interested reader is referred to~\cite{HN13} and~\cite{T13} for further information. 

Most of the constructions of $(u,m,\e,s)$-nets and $(u,\e,s)$-sequences are based on the digital construction method 
introduced by Niederreiter in~\cite{N87}. Let us first outline the digital construction method 
for $(u,m,\e,s)$-nets. Let $s\ge 1$ be a given dimension, let $q$ be a prime power, and consider the finite field $\F_q$ with $q$ elements.  
Furthermore, put $Z_q:=\{0,1,\ldots,q-1\} \subset \Z$. Choose bijections $\eta_r:Z_q\To\Field_q$ for 
all integers $0\le r\le m-1$ and bijections $\kappa_{i,j}:\Field_q\To Z_q$ for $1\le i\le s$ and $1\le j\le m$. Furthermore, we choose 
$m\times m$ \emph{generating matrices} $C_1,\ldots,C_s$ over $\Field_q$. For $n\in\{0,1,\ldots,q^m-1\}$, let 
$n=\sum_{v=0}^{m-1}n_v q^v$ be the base $q$ expansion of $n$. For $1\le i\le s$, we compute the matrix-vector product 
$$ C_i \cdot (\eta_0 (n_0),\eta_1 (n_1),\ldots,\eta_{m-1}(n_{m-1}))^\top=:\left(y_{n,1}^{(i)},y_{n,2}^{(i)},\ldots,y_{n,m}^{(i)}\right)^\top$$ and then we put
$$x_n^{(i)}:=\sum_{j=1}^m \kappa_{i,j} \left(y_{n,j}^{(i)}\right) q^{-j}.$$
Finally, we put $\x_n:=(x_n^{(1)},\ldots,x_n^{(s)}) \in [0,1)^s$ for $0\le n\le q^m-1$. Then the point set consisting of the points $\x_0,\x_1,\ldots,\x_{q^m-1}$ is 
called a \emph{digital net over} $\Field_q$. 
Regarding the parameters of a digital net, we recall the following proposition from~\cite{HN13}.

\begin{proposition}\label{propdignet}
The matrices $C_1,\ldots,C_s\in\Field_q^{m\times m}$ generate a digital $(u,m,\e,s)$-net over $\F_q$ if and only if, for any nonnegative 
integers $d_1,\ldots,d_s$ with $e_i | d_i$ for $1\le i\le s$ and 
$d_1+\cdots + d_s\le m-u$, the collection of the $d_1+ \cdots +d_s$ vectors obtained by taking the first $d_i$ rows of 
$C_i$ for $1\le i\le s$ is linearly independent over $\Field_q$. 
\end{proposition}

For the digital construction of a $(u,\e,s)$-sequence, we choose bijections $\eta_r:Z_q\To\Field_q$ for 
all integers $r\ge 0$, satisfying $\eta_r(0)=0$ for all sufficiently large $r$, and bijections $\kappa_{i,j}:\Field_q\To Z_q$ for $1\le i\le s$ and $j\ge 1$. 
Furthermore, we choose $\infty \times \infty$ \emph{generating matrices} $C_1,\ldots,C_s$ over $\F_q$ (by an $\infty \times \infty$ matrix over $\F_q$ 
we mean a matrix over $\F_q$ with denumerably many rows and columns). 
For an integer $n\ge 0$, let $n=\sum_{v=0}^{\infty}n_v q^v$ be the base $q$ expansion of $n$.  
For $1\le i\le s$, we compute the matrix-vector product 
$$ C_i \cdot (\eta_0 (n_0),\eta_1 (n_1),\ldots)^\top=:\left(y_{n,1}^{(i)},y_{n,2}^{(i)},\ldots\right)^\top$$ and then we put
$$x_n^{(i)}:=\sum_{j=1}^\infty \kappa_{i,j} \left(y_{n,j}^{(i)}\right) q^{-j}.$$
Finally, we put $\x_n:=(x_n^{(1)},\ldots,x_n^{(s)}) \in [0,1]^s$ for $n=0,1,\ldots$ . Then the \se $\x_0,\x_1,\ldots$ is called a \emph{digital 
sequence over} $\Field_q$. As for digital nets, 
the properties of the matrices $C_1,\ldots,C_s$ are intimately related to the parameters of a digital sequence; the following proposition is also 
due to~\cite{HN13}. 

\begin{proposition}\label{propdigseq}
The $\infty \times \infty$ matrices $C_1,\ldots,C_s$ over $\F_q$ generate a digital $(u,\e,s)$-sequence over $\F_q$ if and only if, for every integer 
$m>u$, the left upper $m\times m$ submatrices $C_1^{(m)},\ldots,C_s^{(m)}$ of the 
$C_i$ generate a digital $(u,m,\e,s)$-net over $\F_q$.
\end{proposition}

The problem of how to find $(u,m,\e,s)$-nets with good parameter configurations is non-trivial. One way to tackle this question is to 
consider so-called propagation rules. 
A propagation rule for (digital) nets is a rule that, from one (digital) net or several (digital) nets, produces a (digital) net with new parameters, 
and similarly for (digital) sequences. The theory of propagation rules for classical $(t,m,s)$-nets 
and $(t,s)$-sequences has attracted much interest in the past and there are many results to be found in, e.g., 
\cite{N05}, \cite{NP01}, \cite{NX98}, \cite{NX02} and related papers. We also refer to 
the database MinT (\cite{MinT}), where information on all relevant propagation rules for $(t,m,s)$-nets and $(t,s)$-sequences 
and the resulting parameter configurations can be found. We also remark that there 
exist propagation rules for so-called higher-order nets and sequences as introduced by Dick (see, e.g., \cite{BDP11}, \cite{DK10}, and \cite{DP}).

In the present paper, we study to which extent it is possible to find propagation rules for the new concepts of $(u,m,\e,s)$-nets and 
$(u,\e,s)$-sequences. We will first present propagation rules for (digital) $(u,m,\e,s)$-nets in Section~\ref{secpr}. Then in Section~\ref{secbc}, 
we consider propagation rules for $(u,\e,s)$-sequences that employ a change of the base. In Section~\ref{secduality}, 
we develop a duality theory for digital $(u,m,\e,s)$-nets, and we show 
in Section~\ref{secad} how this theory can be used for finding new digital $(u,m,\e,s)$-nets.

\section{Propagation rules for (digital) $(u,m,\e,s)$-nets} \label{secpr}

In this section, we derive propagation rules for $(u,m,\e,s)$-nets. We first present results that are valid for arbitrary $(u,m,\e,s)$-nets, 
and then we move on to propagation rules that hold for digital nets.
Our first result generalizes the propagation rule that is called Propagation Rule~1 in~\cite{N05}. We write $\N_0$ for the set of nonnegative integers.

\begin{proposition} \label{prpr}
Let $b \ge 2$, $m \ge 0$, $s \ge 1$, and $u\ge 0$ be integers and let $\e =(e_1,\ldots,e_s) \in \N^s$. If a $(u,m,\e,s)$-net in base $b$ is given, then
for every integer $k$ with $u \le k \le m$ such that $m-k$ is a linear combination of $e_1,\ldots,e_s$ with coefficients from $\N_0$, we can construct
a $(u,k,\e,s)$-net in base $b$.
\end{proposition}

\begin{proof}
Let the point set $\mathcal{P}$ be a $(u,m,\e,s)$-net in base $b$. By assumption, we can write $m-k=\sum_{i=1}^s f_ie_i$ with $f_i \in \N_0$ for $1 \le i \le s$.
Consider the interval
$$
E=\prod_{i=1}^s [0,b^{-f_ie_i}) \subseteq [0,1]^s.
$$
Then $\lambda_s(E)=b^{k-m} \ge b^{u-m}$, and so the definition of a $(u,m,\e,s)$-net in base $b$ implies that $E$ contains exactly $b^k$ points of the point set
$\mathcal{P}$. Let these $b^k$ points be
$$
\x_n=(x_n^{(1)},\ldots,x_n^{(s)}) \in E \qquad \mbox{for } n=1,\ldots,b^k.
$$
Now we define the points
$$
\y_n=\left(b^{f_1e_1}x_n^{(1)},\ldots,b^{f_se_s}x_n^{(s)} \right) \in [0,1)^s \qquad \mbox{for } n=1,\ldots,b^k.
$$
We will show that the points $\y_n$, $n=1,\ldots,b^k$, form a $(u,k,\e,s)$-net in base $b$. Let $J \subseteq [0,1]^s$ be an interval of the form
$$
J=\prod_{i=1}^s [a_ib^{-d_i},(a_i+1)b^{-d_i})
$$
with $a_i,d_i \in \N_0$, $a_i < b^{d_i}$, and $e_i|d_i$ for $1 \le i \le s$, and with $\lambda_s(J) \ge b^{u-k}$. Then $\y_n \in J$ if and only if
$$
\x_n \in J^{\prime} := \prod_{i=1}^s [a_ib^{-f_ie_i-d_i},(a_i+1)b^{-f_ie_i-d_i}) \subseteq E.
$$
Note that $\lambda_s(J^{\prime})=b^{k-m} \lambda_s(J) \ge b^{u-m}$. Again by the definition of a $(u,m,\e,s)$-net in base $b$, the number of points $\x_n \in J^{\prime}$ 
is $b^m \lambda_s(J^{\prime})=b^k \lambda_s(J)$, and so the number of points $\y_n \in J$ is $b^k \lambda_s(J)$, which is the desired property.
\end{proof}

The following proposition generalizes a well-known result on $(t,m,s)$-nets and $(t,s)$-sequences from~\cite{N87} (see also \cite[Lemma~4.22]{N92}).

\begin{proposition} \label{prsn}
Let $b \ge 2$, $s \ge 1$, and $u \ge 0$ be integers and let $\e =(e_1,\ldots,e_s) \in \N^s$. If a $(u,\e,s)$-sequence in base $b$ is given, then for every
integer $m \ge u$ we can construct a $(u,m,\e^{\prime},s+1)$-net in base $b$, where $\e^{\prime}=(1,e_1,\ldots,e_s) \in \N^{s+1}$.
\end{proposition}

\begin{proof}
Let $\x_0,\x_1,\ldots$ be a $(u,\e,s)$-\se in base $b$. We fix an integer $m \ge u$ and define the points
$$
\y_n=(nb^{-m},[\x_n]_{b,m}) \in [0,1)^{s+1} \qquad \mbox{for } n=0,1,\ldots,b^m-1.
$$
We claim that these points form a $(u,m,\e^{\prime},s+1)$-net in base $b$. Consider an interval $J^{\prime} \subseteq [0,1]^{s+1}$ of the form
$$
J^{\prime}=\prod_{i=1}^{s+1} [a_ib^{-d_i},(a_i+1)b^{-d_i})
$$
with $a_i,d_i \in \N_0$ and $a_i < b^{d_i}$ for $1 \le i \le s+1$, $e_i|d_i$ for $2 \le i \le s+1$, and $\lambda_{s+1}(J^{\prime}) \ge b^{u-m}$, that is,
$\sum_{i=1}^{s+1} d_i \le m-u$. We have $\y_n \in J^{\prime}$ if and only if $a_1b^{m-d_1} \le n < (a_1+1)b^{m-d_1}$ and
$$
[\x_n]_{b,m} \in J := \prod_{i=2}^{s+1} [a_ib^{-d_i},(a_i+1)b^{-d_i}).
$$
Now $m-d_1 \ge u+\sum_{i=2}^{s+1} d_i \ge u$, and so the definition of a $(u,\e,s)$-\se in base $b$ implies that the points $[\x_n]_{b,m}$ with $a_1b^{m-d_1}
\le n < (a_1+1)b^{m-d_1}$ form a $(u,m-d_1,\e,s)$-net in base $b$. Since
$$
\lambda_s(J)=b^{-(d_2+ \cdots + d_{s+1})} \ge b^{u-m+d_1},
$$
it follows that the number of points $\y_n$, $0 \le n \le b^m-1$, contained in the interval $J^{\prime}$ is equal to $b^{m-d_1} \lambda_s(J)=b^m \lambda_{s+1}
(J^{\prime})$, and so we are done.
\end{proof}

Before we proceed to the formulation of propagation rules for digital $(u,m,\e,s)$-nets, we give the following definition of a 
$(d,m,\e,s)$-system. For the corresponding definition for classical $(t,m,s)$-nets, see, e.g., \cite{NP01} and~\cite{NX98}.

\begin{definition} \label{defsystem} {\rm
Let $q$ be a prime power, let $m \ge 1$ and $s \ge 1$ be integers, let $d$ be an integer with $0\le d\le m$, and let 
$\e=(e_1,\ldots,e_s)\in \N^s$. The system 
$$A=\left\{\bsa_j^{(i)}\in\Field_q^m: 1\le i\le s, 1\le j\le m\right\}$$
of vectors is called a $(d,m,\e,s)$-\emph{system over} $\Field_q$ if for any choice of 
nonnegative integers $d_1,\ldots,d_s$ with $e_i |d_i$ for $1\le i\le s$ and $\sum_{i=1}^s d_i\le d$, the 
vectors $\bsa_j^{(i)}$, $1\le j\le d_i$, $1\le i\le s$, are linearly independent over $\Field_q$ (this property is 
assumed to be trivially satisfied for $d=0$). } 
\end{definition}

The next lemma is analogous to \cite[Lemma~3]{NX98}. 
For an $m \times m$ matrix $C_i$ over $\F_q$ and $1 \le j \le m$, we write $\bsc_j^{(i)}$ for the $j$th
row vector of $C_i$.

\begin{lemma} \label{lemnetsystem}
Let $q$ be a prime power, let $m \ge 1$, $s \ge 1$, and $u \ge 0$ be integers, and let $\e \in \N^s$. 
A digital net over $\Field_q$ with $m \times m$ generating matrices $C_1,\ldots,C_s$ over $\F_q$ 
is a digital $(u,m,\e,s)$-net over $\F_q$ if and only if the system $\{\bsc_j^{(i)} \in \F_q^m: 1 \le i \le s, 1 \le j \le m\}$ 
of row vectors of the matrices 
$C_1,\ldots,C_s$ is an $(m-u,m,\e,s)$-system over $\Field_q$.
\end{lemma}

\begin{proof}
The result follows immediately from the definition of a $(d,m,\e,s)$-system and Proposition~\ref{propdignet}. 
\end{proof}

Lemma~\ref{lemnetsystem} enables us to show the following propagation rule, which is a generalization of \cite[Theorem~10]{NX98} (also 
referred to as Direct Product Rule or Propagation Rule 4 in \cite{N05}).

\begin{theorem}\label{thmdirectproduct}
Let $q$ be a prime power, let $m_1,m_2 \ge 1$, $s_1,s_2 \ge 1$, and $u_1,u_2 \ge 0$ be integers, and let $\e= (e_1,\ldots,e_{s_1})\in \N^{s_1}$, 
${\bf f}=(f_1,\ldots,f_{s_2})\in \N^{s_2}$. If a digital $(u_1,m_1,\e,s_1)$-net over $\F_q$ and a digital 
$(u_2,m_2,{\bf f},s_2)$-net over $\Field_q$ are given, 
then we can construct a digital $(u,m_1 + m_2,(\e,{\bf f}),s_1 + s_2)$-net over $\Field_q$ with 
$$u=\max\{m_1 + u_2, m_2 + u_1\}$$
and 
$$(\e,{\bf f}) =(e_1,\ldots, e_{s_1}, f_1,\ldots, f_{s_2}) \in \N^{s_1+s_2}.$$
\end{theorem}

\begin{proof} 
By Lemma~\ref{lemnetsystem} it suffices to show that a $(d_1,m_1,(e_1,\ldots,e_{s_1}),s_1)$-system over $\F_q$ and a 
 $(d_2,m_2,(f_1,\ldots,f_{s_2}),s_2)$-system over $\Field_q$ yield a $(d,m_1 + m_2,(e_1,\ldots, e_{s_1},f_1,\ldots,f_{s_2}), s_1 + s_2)$-system 
over $\Field_q$ with $d=\min\{d_1,d_2\}$. Let
$$A=\left\{\bsa_j^{(i)}\in\Field_q ^{m_1}: 1\le i\le s_1, 1\le j\le m_1\right\}$$
be a $(d_1,m_1,(e_1,\ldots,e_{s_1}),s_1)$-system 
and 
$$B=\left\{\bsb_j^{(i)}\in\Field_q ^{m_2}: 1\le i\le s_2, 1\le j\le m_2\right\}$$
be a $(d_2,m_2,(f_1,\ldots,f_{s_2}),s_2)$-system over $\Field_q$, where we assume, without loss of 
generality, that $m_2\ge m_1$. Now we define the system 
$$C=\left\{\bsc_j^{(i)}\in\Field_q ^{m_1+m_2}: 1\le i\le s_1 + s_2, 1\le j\le m_1+m_2\right\}$$
with
$$\bsc_j^{(i)} :=\begin{cases} 
 		(\bsa_j^{(i)},\bszero) & \mbox{for $1\le i\le s_1,\ 1\le j\le m_1$,}\\
		(\bszero,\bsb_j^{(i-s_1)}) & \mbox{for $s_1< i\le s_1 + s_2,\ 1\le j\le m_1$,}\\
		\bszero & \mbox{for $1\le i\le s_1 + s_2,\ m_1 < j\le m_1 + m_2$.}
                \end{cases}$$
We now show that $C$ is indeed a $(d,m_1 + m_2,(e_1,\ldots, e_{s_1},f_1,\ldots,f_{s_2}), s_1 + s_2)$-system over $\Field_q$.
To this end, choose nonnegative integers $r_1,\ldots,r_{s_1},r_{s_1 +1},\ldots, r_{s_1+s_2}$ with $e_i| r_i$ for $1\le i\le s_1$ and $f_{i-s_1} | r_i$ for 
$s_1 < i\le s_1 + s_2$, and $\sum_{i=1}^{s_1 + s_2} r_i \le d$. We need to check that the vectors 
$$(\bsa_j^{(i)},\bszero),\ 1\le j\le r_i,\ 1\le i\le s_1,$$
and 
$$(\bszero,\bsb_j^{(i-s_1)}),\ 1\le j\le r_i,\ s_1< i\le s_1 + s_2,$$
are linearly independent over $\Field_q$. But this property of $C$ follows immediately by the assumptions made on $A$ and $B$ and the fact 
that $d$ does not exceed $d_1$ or $d_2$.
\end{proof}

The following corollary, which follows immediately from Theorem~\ref{thmdirectproduct} by induction, is a generalization of \cite[Corollary~3]{NX98}. 

\begin{corollary}
Let $q$ be a prime power, let $m_1,\ldots,m_n \ge 1$, $s_1,\ldots,s_n \ge 1$, and $u_1,\ldots,u_n\ge 0$ be integers, 
and let $\e_k= (e_1,\ldots,e_{s_k})\in \N^{s_k}$
for $1\le k\le n$. If digital $(u_k,m_k,\e_k,s_k)$-nets over $\Field_q$ for $1\le k\le n$ are given, then there exists a digital 
$$\left(u,\sum_{k=1}^{n} m_k,\e,\sum_{k=1}^n s_k \right)\mbox{-net}$$
over $\Field_q$ with 
$$u=\sum_{k=1}^n m_k -\min_{1\le k\le n}(m_k-u_k)$$
and 
$$\e=(e_1^{(1)},\ldots, e_{s_1}^{(1)}, e_1^{(2)},\ldots,e_{s_2}^{(2)},\ldots\ldots, e_1^{(n)},
\ldots,e_{s_n}^{(n)}) \in \N^{s_1 + \cdots + s_n}.$$
\end{corollary}

Up to now, we have dealt only with propagation rules where one or several point sets or sequences in a certain base were given, and we constructed 
a new point set in the same base. However, it is known that one can derive propagation rules by making a transition from one base to another. 
These propagation rules are called base-change propagation rules.
We now show the following theorem, which is a generalization of Theorem~9 in \cite{NX98} (or Propagation Rule~7 in \cite{N05}). 

\begin{theorem} \label{thmbasechange}
Let $q$ be a prime power, let $m\ge 1$, $s\ge 1$, $u \ge 0$, and $r\ge 1$ be integers, and let 
$\e=(e_1,\ldots,e_s)\in\NN^s$. If there exists a digital $(u,m,\e,s)$-net over $\Field_{q^r}$, 
then there exists a digital $((r-1)m+u,rm,{\bf f},rs)$-net over $\Field_q$, where
$${\bf f}=(\underbrace{e_1,\ldots,e_1}_{\mbox{$r$ times}},
\underbrace{e_2,\ldots,e_2}_{\mbox{$r$ times}},\ldots,\underbrace{e_s,\ldots,e_s}_{\mbox{$r$ times}}).$$
\end{theorem}

\begin{proof}  
By Lemma~\ref{lemnetsystem} it suffices to show that we can obtain a $(d,rm,{\bf f},rs)$-system over $\Field_q$ from a
$(d,m,\e,s)$-system over $\Field_{q^r}$. To this end, let
$$ A=\left\{\bsa_j^{(i)}\in\Field_{q^r}^m: 1\le i\le s, 1\le j\le m\right\} $$
be a $(d,m,\e,s)$-system over $\Field_{q^r}$. Choose an ordered basis $\beta_1,\ldots,\beta_r$ of $\Field_{q^r}$ over $\Field_q$ and an 
$\Field_q$-linear isomorphism $\varphi: \Field_{q^r}^m\To\Field_q^{rm}$. Define now a system 
$$ B=\left\{\bsb_j^{(h)}\in\Field_{q}^{rm}: 1\le h\le rs, 1\le j\le rm\right\}$$
by setting, for $1\le i\le s$ and $1\le k\le r$,
$$\bsb_j^{((i-1)r+k)}:=\begin{cases}
                     \varphi (\beta_k\bsa_j^{(i)}) & \mbox{for}\ 1\le j\le m,\\
                     \bszero & \mbox{for}\ m< j\le rm.
                    \end{cases}$$
We claim that $B$ is a $(d,rm,{\bf f},rs)$-system over $\Field_q$. Indeed, choose nonnegative integers $d_{i,k}$ for $1 \le i \le s$, $1 \le k \le r$, with 
the properties $e_i | d_{i,k}$ for $1\le i\le s$, $1\le k\le r$,  and 
$$ \sum_{i=1}^s \sum_{k=1}^r d_{i,k}\le d. $$
Suppose that 
\begin{equation}\label{eqsupposezero}
 \sum_{i=1}^s \sum_{k=1}^r\sum_{j=1}^{d_{i,k}} c_j^{(i,k)} \bsb_j^{((i-1)r +k)}=\bszero\in\Field_q^{rm},
\end{equation}
where all $c_j^{(i,k)}\in\Field_q$. For $1\le i\le s$ define
$ d_i:=\max_{1\le k\le r} d_{i,k}, $
and note that $d_i\le m$, as $d_{i,k}\le d\le m$, and that $e_i | d_i$. 

For $1\le i\le s$ and $1\le k\le r$, let $h_j^{(i,k)}:=1\in\Field_q$ for $1\le j\le d_{i,k}$ and 
$h_j^{(i,k)}=0\in\Field_q$ for $d_{i,k} < j\le d_i$. Then we can write~\eqref{eqsupposezero} as 
$$\sum_{i=1}^s \sum_{k=1}^r\sum_{j=1}^{d_{i}} h_j^{(i,k)} c_j^{(i,k)} \varphi(\beta_k \bsa_j^{(i)})=\bszero\in\Field_q^{rm},$$
from which we conclude, due to the properties of $\varphi$, that 
$$ \sum_{i=1}^s \sum_{j=1}^{d_i} \gamma_j^{(i)} \bsa_j^{(i)}=\bszero\in\Field_{q^r}^m$$
with
$$\gamma_j^{(i)}=\sum_{k=1}^r h_j^{(i,k)} c_j^{(i,k)} \beta_k \in\Field_{q^r}. $$
Note now that 
$$\sum_{i=1}^s d_i \le \sum_{i=1}^s \sum_{k=1}^r d_{i,k}\le d.$$
As we have $e_i | d_i$ for $1\le i\le s$, and since we assumed $A$ to be a $(d,m,\e,s)$-system over $\Field_{q^r}$, 
we must have $\gamma_j^{(i)}=0$ for $1\le j\le d_i$, $1\le i\le s$, and therefore 
also $h_j^{(i,k)}c_j^{(i,k)}=0$ for $1\le j\le d_i$, $1\le i\le s$, $1\le k\le r$. Therefore we see that all coefficients $c_j^{(i,k)}$  
in~\eqref{eqsupposezero} must be equal to $0$.
\end{proof}

We also generalize the following propagation rule that is called ``base reduction for projective spaces'' in~\cite{MinT},
by proving the following result.

\begin{theorem} \label{thmprojective}
Let $q$ be a prime power, let $m\ge 1$, $s\ge 1$, $u \ge 0$, and $r\ge 2$ be integers, and let 
$\e=(e_1,\ldots,e_s)\in\NN^s$. If there exists a digital $(u,m,\e,s)$-net 
over $\Field_{q^r}$, then there exists a digital $((r-1)m-(r-1)+u,rm-(r-1),\e,s)$-net over $\Field_q$.
\end{theorem}

\begin{proof} 
The result can be shown by an adaptation of the proof of Theorem~\ref{thmbasechange}. 
Let a digital $(u,m,\e,s)$-net over $\Field_{q^r}$ be given and let $C_1,\ldots,C_s$ be its generating matrices. 
Note that the linear independence conditions in the definition of a digital $(u,m,\e,s)$-net stay unchanged if we 
multiply a row of a matrix $C_i$ with some nonzero element of $\Field_{q^r}$. Doing so, we can obtain generating 
matrices $A_1,\ldots,A_s$ over $\Field_{q^r}$ which also generate a digital $(u,m,\e,s)$-net over $\F_{q^r}$ and for which the first column 
of each $A_i$, $1\le i\le s$, consists only of zeros and ones. Let $\bsa_{j}^{(i)}$, $1\le j\le m$, $1\le i\le s$, be the row vectors 
of the matrices $A_i$. By Lemma~\ref{lemnetsystem}, the system
$$A=\{\bsa_j^{(i)}\in\Field_{q^r}^m: 1\le i\le s, 1\le j\le m\}$$
is an $(m-u,m,\e,s)$-system over $\Field_{q^r}$. Let now $\psi$ be an $\Field_q$-linear isomorphism from $\Field_{q^r}$ to $\Field_q^r$ 
such that $\psi (1)= (0,\ldots,0,1)\in\Field_q^r$. For a vector $\bsa\in\Field_{q^r}^m$ with $\bsa=(\alpha_1,\ldots,\alpha_m)$, put
$\varphi (\bsa)=(\psi (\alpha_1),\ldots,\psi (\alpha_m))\in\Field_q^{rm}$. 

Define now a new system 
$$B=\{\bsb_j^{(i)}\in\Field_q^{rm}: 1\le i\le s, 1\le j\le rm\} $$
by setting, for $1\le i\le s$,
$$\bsb_j^{(i)}=\begin{cases}
                \varphi (\bsa_j^{(i)}) & \mbox{for $1\le j\le m$,}\\
                \bszero & \mbox{for $m< j\le rm$.}
               \end{cases}$$
We are now going to show that $B$ is an $(m-u,rm,\e,s)$-system over $\Field_q$. Choose nonnegative integers $d_1,\ldots,d_s$ with 
the properties $e_i | d_{i}$ for $1\le i\le s$  and $\sum_{i=1}^s d_{i}\le m-u$.
Suppose that 
\begin{equation}\label{eqsupposezero2}
 \sum_{i=1}^s \sum_{j=1}^{d_{i}} c_j^{(i)} \bsb_j^{(i)}=\bszero\in\Field_q^{rm},
\end{equation}
where all $c_j^{(i)}\in\Field_q$. Then we can write~\eqref{eqsupposezero2} as 
$$\sum_{i=1}^s \sum_{j=1}^{d_{i}} c_j^{(i)} \varphi(\bsa_j^{(i)})=\bszero\in\Field_q^{rm},$$
from which we conclude, due to the properties of $\varphi$, that 
$$ \sum_{i=1}^s \sum_{j=1}^{d_i} c_j^{(i)} \bsa_j^{(i)}=\bszero\in\Field_{q^r}^m.$$
As we have $\sum_{i=1}^s d_{i}\le m-u$ and $e_i | d_i$ for $1\le i\le s$, and since we assumed $A$ to be an $(m-u,m,\e,s)$-system over $\Field_{q^r}$, 
we must have $c_j^{(i)}=0$ for $1\le j\le d_i$, $1\le i\le s$, in~\eqref{eqsupposezero2}. Consequently, 
$B$ is indeed an $(m-u,rm,\e,s)$-system over $\Field_q$.

Note now that the first $r-1$ coordinates of each $\bsb_j^{(i)}$ are equal to zero, due to the choices of $\psi$ and $\varphi$ and due to the 
fact that the first coordinates of the $\bsa_j^{(i)}$ are all either zero or one. Furthermore, note that 
$\bsb_{rm-(r-2)}^{(i)},\bsb_{rm-(r-3)}^{(i)},\ldots,\bsb_{rm}^{(i)}$ are all $\bszero$, as we assumed $r\ge 2$. 

We now remove $\bsb_{rm-(r-2)}^{(i)},\bsb_{rm-(r-3)}^{(i)},\ldots,\bsb_{rm}^{(i)}$ from $B$ and 
discard for each of the remaining $\bsb_j^{(i)}\in B$ its first $r-1$ coordinates. In this way we end up with a system 
$$D=\{\bsd_j^{(i)}\in\Field_q^{rm - (r-1)}: 1\le i\le s, 1\le j\le rm - (r-1)\},$$
where the $\bsd_j^{(i)}$ are the projections of the original $\bsb_j^{(i)}$ onto their last $rm- (r-1)$ coordinates. However, the linear 
independence properties of the vectors in $D$ are the same as those of the vectors in $B$, since we removed only zeros to derive $D$ from $B$. 
Hence, $D$ is an $(m-u,rm- (r-1),\e,s)$-system over $\Field_q$. By Lemma~\ref{lemnetsystem}, the vectors in $D$ generate a 
digital $((r-1)m-(r-1)+u,rm-(r-1),\e,s)$-net over $\Field_{q}$, as claimed.
\end{proof}

Now we establish the digital analog of the propagation rule in Proposition~\ref{prpr}. In the classical case of digital $(t,m,s)$-nets, this propagation rule
was shown in~\cite{SW} (see also \cite[Theorem~4.60]{DP}).

\begin{proposition} \label{prprd}
Let $q$ be a prime power, let $m \ge 1$, $s \ge 1$, and $u \ge 0$ be integers, and let $\e =(e_1,\ldots,e_s) \in \N^s$. 
If a digital $(u,m,\e,s)$-net over $\F_q$
is given, then for every integer $k$ with $\max (1,u) \le k \le m$ such that $m-k$ is a linear combination of $e_1,\ldots,e_s$ 
with coefficients from $\N_0$,
we can construct a digital $(u,k,\e,s)$-net over $\F_q$.
\end{proposition}

\begin{proof}
It suffices to consider the case where $m-k$ is divisible by some $e_i$, since we can then proceed by induction. By using a permutation of the 
coordinates, we can
assume that $e_s$ divides $m-k$. In the generating matrix $C_s$ of the given digital $(u,m,\e,s)$-net over $\F_q$, the first $e_s \lfloor (m-u)/e_s \rfloor$ row vectors
are linearly independent over $\F_q$, whereas the remaining row vectors do not matter (compare with Lemma~\ref{lemnetsystem}). 
Therefore we can assume that the row
vectors of $C_s$ form a basis of $\F_q^m$, and by changing the coordinate system in $\F_q^m$ we can take $C_s$ to be the antidiagonal matrix $E_m^{\prime}$ in the
proof of \cite[Theorem~4.60]{DP}. Now we can imitate that proof (but note that $n$ in that proof plays the role of our $k$ and that the condition $d_1 + \cdots + d_s
=n-t$ is now replaced by $d_1 + \cdots + d_s \le k-u$). The only point we need to observe is that the number $m-k$ of auxiliary unit vectors in the displayed scheme
of vectors in \cite[p.~157]{DP} (namely the unit vectors with the single coordinate $1$ between position $k+1$ and position $m$) must be divisible by $e_s$. But this
is guaranteed by our assumption.
\end{proof} 

Next we show the digital analog of Proposition~\ref{prsn}. First we reformulate Proposition~\ref{propdigseq} in the language of $(d,m,\e,s)$-systems. 

\begin{lemma} \label{lemsnd}
Let $q$ be a prime power, let $s \ge 1$ and $u \ge 0$ be integers, and let $\e \in \N^s$. Then the $\infty \times \infty$ matrices $C_1,\ldots,C_s$ over $\F_q$ generate
a digital $(u,\e,s)$-sequence over $\F_q$ if and only if, for every integer $m \ge \max (1,u)$, the system of row vectors of the matrices $C_1^{(m)},\ldots,C_s^{(m)}$ 
is an $(m-u,m,\e,s)$-system over $\F_q$.
\end{lemma}

\begin{proposition} \label{prsnd}
Let $q$ be a prime power, let $s \ge 1$ and $u \ge 0$ be integers, and let $\e =(e_1,\ldots,e_s) \in \N^s$. If a digital $(u,\e,s)$-sequence over $\F_q$ is given, 
then for every integer $m \ge \max (1,u)$ we can construct a digital $(u,m,\e^{\prime},s+1)$-net over $\F_q$, where $\e^{\prime}=(1,e_1,\ldots,e_s) \in \N^{s+1}$.
\end{proposition}

\begin{proof}
Let $D_1,\ldots,D_s$ be $\infty \times \infty$ generating matrices over $\F_q$ of the given digital $(u,\e,s)$-sequence over $\F_q$. Fix an integer $m \ge \max (1,u)$.
We define $m \times m$ generating matrices $C_1,\ldots,C_{s+1}$ over $\F_q$ by letting $C_1$ be a nonsingular right lower triangular matrix over $\F_q$ and by setting
$C_i=D_{i-1}^{(m)}$ for $2 \le i \le s+1$. By Lemma~\ref{lemnetsystem} it suffices to show that the system of row vectors of the matrices $C_1,\ldots,C_{s+1}$ is an
$(m-u,m,\e^{\prime},s+1)$-system over $\F_q$. To this end, we choose $d_1,\ldots,d_{s+1} \in \N_0$ with $e_i|d_{i+1}$ for $1 \le i \le s$ and $\sum_{i=1}^{s+1} d_i
\le m-u$, and we have to prove that the row vectors $\bsc_j^{(i)}$, $1 \le j \le d_i$, $1 \le i \le s+1$, of the matrices $C_1,\ldots,C_{s+1}$ are linearly independent
over $\F_q$. If $d_1=m-u$, then $d_i=0$ for $2 \le i \le s+1$, and we are done since $C_1$ is nonsingular. Hence we can assume that $d_1 < m-u$. Suppose that we have  
\begin{equation} \label{eqla}
\sum_{i=1}^{s+1} \sum_{j=1}^{d_i} b_j^{(i)} \bsc_j^{(i)} = \bszero \in \F_q^m
\end{equation}
with all $b_j^{(i)} \in \F_q$. Let $\pi : \F_q^m \to \F_q^{m-d_1}$ be the projection to the first $m-d_1$ coordinates of a vector in $\F_q^m$. Then $\pi (\bsc_j^{(1)})=
\bszero \in \F_q^{m-d_1}$ for $1 \le j \le d_1$ since $C_1$ is a right lower triangular matrix, and so applying $\pi$ to~\eqref{eqla} we obtain
\begin{equation} \label{eqlb}
\sum_{i=2}^{s+1} \sum_{j=1}^{d_i} b_j^{(i)} \pi(\bsc_j^{(i)})=\bszero \in \F_q^{m-d_1}.
\end{equation}
Now the vectors $\pi(\bsc_j^{(i)})$, $1 \le j \le d_i$, $2 \le i \le s+1$, are row vectors of the matrices $D_1^{(m-d_1)},\ldots,D_s^{(m-d_1)}$. The system of all row vectors
of these matrices forms an $(m-d_1-u,m-d_1,\e,s)$-system over $\F_q$ by Lemma~\ref{lemsnd}. By observing that $\sum_{i=2}^{s+1} d_i \le m-d_1-u$, we conclude from~\eqref{eqlb}
that $b_j^{(i)}=0$ for $1 \le j \le d_i$, $2 \le i \le s+1$, and returning to~\eqref{eqla} we see that $b_j^{(1)}=0$ for $1 \le j \le d_1$.
\end{proof}

\section{Base-change rules for $(u,\e,s)$-sequences} \label{secbc}

Now we apply the idea of a base change to $(u,\e,s)$-sequences. Note that the results in this section are not propagation rules 
in the sense that they yield a new sequence, but they are statements on how we can view a given sequence with respect to different bases. 
We first need the following auxiliary result.

\begin{lemma} \label{lebc}
Let $b \ge 2$, $g \ge 1$, $m \ge 1$, and $s \ge 1$ be integers. Let the subinterval $J$ of $[0,1]^s$ be of the form
\begin{equation} \label{eqj}
J=\prod_{i=1}^s [a_i b^{-gf_i}, (a_i+1) b^{-gf_i})
\end{equation}
with $a_i,f_i \in \N_0$, $a_i < b^{gf_i}$, and $f_i \le m$ for $1 \le i \le s$. Then for $\x \in [0,1]^s$ we have $[\x]_{b^g,m} \in J$ if and only if
$[\x]_{b,gm} \in J$.
\end{lemma}

\begin{proof}
It suffices to show the lemma for $s=1$. Thus, let $J=[a b^{-gf},(a+1) b^{-gf})$ with $a,f \in \N_0$, $a < b^{gf}$, and $f \le m$. Then $[x]_{b^g,m} \in J$
means that the first $f$ digits of $x \in [0,1]$ in base $b^g$ are prescribed. Now $f$ digits in base $b^g$ correspond to $gf$ digits in base $b$, and so
$[x]_{b^g,m} \in J$ is equivalent to $[x]_{b,gf} \in J$. Finally, $f \le m$ implies that we have $[x]_{b,gf} \in J$ if and only if $[x]_{b,gm} \in J$.
\end{proof}

\begin{theorem} \label{thbc}
Let $b \ge 2$, $s\ge 1$, and $u \ge 0$ 
be integers and let $\e =(e_1,\ldots,e_s) \in \N^s$. Then any $(u,\e,s)$-\se in base $b$ is a $(\lceil u/g \rceil,
s)$-\se in base $b^g$, where $g$ is the least common multiple of $e_1,\ldots,e_s$.
\end{theorem}

\begin{proof}
It suffices to show that any $(gu,\e,s)$-\se in base $b$ is a $(u,s)$-\se in base $b^g$. For then, with $v=\lceil u/g \rceil$, we have $gv \ge u$, hence a 
given $(u,\e,s)$-\se in base $b$ is also a $(gv,\e,s)$-\se in base $b$ by \cite[Remark~1 and Definition~2]{HN13}, and so it is a $(v,s)$-\se in base $b^g$.  

Thus, let $\x_0,\x_1,\ldots$ be a given $(gu,\e,s)$-\se in base $b$, where $g$ is the least common multiple of the components of $\e$. We want to prove that $\x_0,\x_1,\ldots$ is a $(u,s)$-\se in base $b^g$. For given 
integers $k \ge 0$ and $m > u$, we have to show that the points $[\x_n]_{b^g,m}$ with $kb^{gm} \le n < (k+1)b^{gm}$ form a $(u,m,s)$-net in base $b^g$. Take
an elementary interval $J$ in base $b^g$ as in~\eqref{eqj} with $\lambda_s(J)=(b^g)^{u-m}$, that is, with $\sum_{i=1}^s f_i=m-u$. Since $\x_0,\x_1,\ldots$ 
form a $(gu,\e,s)$-\se in base $b$ and since $gm > gu$, it follows that the points $[\x_n]_{b,gm}$ with $kb^{gm} \le n < (k+1)b^{gm}$ form a $(gu,gm,\e,s)$-net
in base $b$. Now $e_i | gf_i$ for $1 \le i \le s$ and $\lambda_s(J)=b^{gu-gm}$, and so the definition of a $(gu,gm,\e,s)$-net in base $b$ implies that the
number of integers $n$ with $kb^{gm} \le n < (k+1)b^{gm}$ and $[\x_n]_{b,gm} \in J$ (or equivalently $[\x_n]_{b^g,m} \in J$ by Lemma~\ref{lebc}) is given by
$b^{gm} \lambda_s(J)=b^{gu}$. This shows that the points $[\x_n]_{b^g,m}$ with $kb^{gm} \le n < (k+1)b^{gm}$ form a $(u,m,s)$-net in base $b^g$.
\end{proof}

\begin{theorem} \label{thbc2}
Let $b \ge 2$, $s\ge 1$, and $u \ge 0$ 
be integers and let $\e =(e_1,\ldots,e_s) \in \N^s$. Furthermore, denote by $h$ the greatest common divisor of $e_1,\ldots,e_s$. 
Then any $(u,\e,s)$-\se in base $b$ is a $(\lceil u/h \rceil, \e/h,s)$-\se in base $b^h$, where we write $\e /h$ for $(e_1/h, e_2/h,\ldots, e_s/h)$.
\end{theorem}

\begin{proof}
Similarly to the proof of Theorem~\ref{thbc}, it is sufficient to show that any $(hu,\e,s)$-\se in base $b$ is a $(u,\e/h,s)$-\se in base $b^h$, with $h$ 
being the greatest common divisor of the components of $\e$. 

Let therefore $\e$ be given and let $h$ be defined as above. Let $\x_0,\x_1,\ldots$ be a given $(hu,\e,s)$-\se in base $b$. 
We want to prove that $\x_0,\x_1,\ldots$ is a $(u,\e /h,s)$-\se in base $b^h$. The argument works in an analogous way to the proof of Theorem~\ref{thbc}, the only 
difference being that, for $m>u$, we consider an elementary interval $J$ in base $b^h$ with $\lambda_s (J)\ge (b^h)^{u-m}$ (that is, with $\sum_{i=1}^s f_i \le m-u$), 
and with $(e_i/h)|f_i$ for  $1\le i\le s$. We can then follow the argument in the proof of Theorem~\ref{thbc}, and we then have $\lambda_s (J)\ge b^{hu-hm}$ and 
$e_i | hf_i$ for $1\le i\le s$, which yields, in exactly the same fashion as above, the desired result. 
\end{proof}

\section{Duality theory} \label{secduality}

We generalize the classical duality theory for digital nets introduced in~\cite{NP01} by developing a duality theory for digital $(u,m,\e,s)$-nets over the finite
field $\F_q$, where $q$ is an arbitrary prime power. Throughout this section, the prime power $q$ and the positive integer $m$ are fixed.

For $\bsa=(a_1,\ldots,a_m)\in\Field_q^m$ and $e\in\NN$, we introduce the weight $v_e (\bsa)$ by
$$ v_e(\bsa)=\begin{cases}
              0 &\mbox{if $\bsa=\bszero$,}\\
              \min\left\{m,e\left\lceil\max\{j: a_j\neq 0\}/e\right\rceil\right\} & \mbox{otherwise.}
             \end{cases}$$
This definition can be extended to $\Field_q^{sm}$ by considering a vector $\bsA\in\Field_q^{sm}$ as the concatenation of 
$s$ vectors of length $m$ each, i.e.,
$$\bsA=(\bsa^{(1)},\ldots,\bsa^{(s)})\in\Field_q^{sm}$$
with $\bsa^{(i)}\in\Field_q^m$ for $1\le i\le s$, 
and by putting, for $\e=(e_1,\ldots,e_s)\in\NN^s$, 
$$ V_{m,\e}(\bsA):=\sum_{i=1}^s v_{e_i} (\bsa^{(i)}).$$

\begin{remark} \label{revm} {\rm
For $\e= {\bf 1}=(1,\ldots,1)$, $V_{m,\e}$ coincides with $V_m$ defined in~\cite{NP01}.}
\end{remark}

\begin{definition} \label{defmd} {\rm
Let $q$ be a prime power, let $s,m\ge 1$ be integers, and let $\cN$ be a linear subspace of $\Field_q^{sm}$. 
Then we define the \emph{minimum distance} $\delta_{m,\e}(\cN)$ of $\cN$ (for $\e\in\NN^s$) as
$$\delta_{m,\e}(\cN):=\begin{cases}
                       \min_{\bsA \in\cN\setminus \{\bszero\}} V_{m,\e} (\bsA) & \mbox{if $\cN\neq\{\bszero\}$},\\
                       sm+1 & \mbox{otherwise.}
                      \end{cases}
$$
}
\end{definition}

It is trivial that we have $\delta_{m,\e} (\cN)\ge 1$ for every linear subspace $\cN$ of $\Field_q^{sm}$. We now show the following result, which is a
generalization of \cite[Proposition~1]{NP01}. 

\begin{proposition}
Let $q$ be a prime power and $s,m\ge 1$ be integers. For any linear subspace $\cN$ of $\Field_q^{sm}$ and for any $\e\in\NN^s$, we have 
$$\delta_{m,\e} (\cN)\le sm -\dim (\cN) + \min_{1\le i\le s} e_i. $$
\end{proposition}

\begin{proof}
If $\cN=\{\bszero\}$, the result is trivial.
If $\cN\neq \{\bszero\}$, let $h:=\dim(\cN)\ge 1$. We write $sm-h+1=km+r$ with integers $0 \le k \le s-1$ and $1 \le r \le m$. Without loss of generality, we can
use a permutation of the $e_i$ and the same permutation for the components $\bsa^{(i)}$ of each $\bsA \in \F_q^{sm}$ such that $\min_{1 \le i \le s} e_i=e_{k+1}$.
Let $\pi:\cN\To\Field_q^h$ be the linear transformation which maps
$\bsA\in\cN$ to the $h$-tuple of the last $h$ coordinates of $\bsA$. If $\pi$ is surjective, then there exists a nonzero 
$\bsA_1\in\cN$ with 
$$\pi (\bsA_1)=(1,0,\ldots,0)\in\Field_q^h.$$
Let $\bsA_1=(\bsb^{(1)},\ldots,\bsb^{(s)})$ with $\bsb^{(i)} \in \F_q^m$ for $1 \le i \le s$. For $1 \le i \le k$ we have the trivial bound $v_{e_i}(\bsb^{(i)})
\le m$. For $i=k+1$ we have
$$
v_{e_{k+1}}(\bsb^{(k+1)}) \le e_{k+1} \lceil r/e_{k+1} \rceil \le r+e_{k+1}-1.
$$
For $k+2 \le i \le s$ we have $v_{e_i}(\bsb^{(i)}) =0$. Therefore
$$
V_{m,\e}(\bsA_1) \le km+r+e_{k+1}-1=sm-h+ \min_{1 \le i \le s} e_i.
$$ 

If $\pi$ is not surjective, then for any nonzero $\bsA_2$ in the kernel of $\pi$ we obtain by a similar argument,
$$V_{m,\e} (\bsA_2)\le sm-h-1 +\min_{1\le i\le s} e_i,$$
and so in both cases the result follows.
\end{proof}

We introduce some additional notation. With a given system of vectors 
$$\left\{\bsc_j^{(i)}\in\Field_q^m: 1\le j\le m, 1\le i\le s\right\},$$
we associate the matrices $C_i$, $1\le i\le s$, as the $m\times m$ matrices with column vectors 
$\bsc_1^{(i)},\ldots,\bsc_m^{(i)}$. Combine these matrices into the $m \times sm$ matrix
$$C=(C_1 | C_2 | \cdots | C_s),$$
so that $C_1,\ldots,C_s$ are submatrices of $C$. By $\cC \subseteq \F_q^{sm}$ we denote the row space of $C$ and by 
$\cC^\perp \subseteq \F_q^{sm}$ the dual space of $\cC$.
We now show the following theorem.

\begin{theorem} \label{thmsystemdual}
Let $q$ be a prime power, let $s,m\ge 1$ be integers, and $\e\in\NN^s$. Furthermore, let the system
$$\left\{\bsc_j^{(i)}\in\Field_q^m: 1\le j\le m, 1\le i\le s\right\}$$
be given. Then with the notation above, the following results hold for an integer $d\in\{0,\ldots,m\}$.
\begin{itemize}
 \item[(a)] If the system is 
 a $(d,m,\e,s)$-system over $\Field_q$, then  $\delta_{m,\e} (\cC^\perp)\ge d$.
 \item[(b)] If the matrices $C_1,\ldots,C_s$ are all nonsingular and if the system is 
 a $(d,m,\e,s)$-system over $\Field_q$, then  $\delta_{m,\e} (\cC^\perp)\ge d+1$.
  \item[(c)] If $\delta_{m,\e} (\cC^\perp)\ge d+1$, then it follows that the system is a 
 $(d,m,\e,s)$-system over $\Field_q$. 
\end{itemize}
\end{theorem}

\begin{proof}
The result is trivial if $\cC=\Field_q^{sm}$, so we can assume that $\cC\subsetneq\Field_q^{sm}$ in the following. Note that then $\dim(\cC^{\perp}) \ge 1$. 
For $\bsA=(\bsa^{(1)},\ldots,\bsa^{(s)})\in\Field_q^{sm}$ with $\bsa^{(i)}=(a_1^{(i)},\ldots,a_m^{(i)})\in\Field_q^m$ 
for $1\le i\le s$, we have
$$\sum_{i=1}^s \sum_{j=1}^m a_j^{(i)}\bsc_j^{(i)}=\bszero\in\Field_q^m$$
if and only if $C \bsA^{\top}=\bszero \in \F_q^m$, i.e., if and only if $\bsA \in \cC^\perp$.

Now let the given system be a $(d,m,\e,s)$-system over $\Field_q$ and consider any nonzero $\bsA\in\cC^\perp$. Then from 
the previous observation we obtain
\begin{equation}\label{eqlindependent}
 \sum_{i=1}^s \sum_{j=1}^m a_j^{(i)} \bsc_j^{(i)}=\bszero\in\Field_q^m.
\end{equation}

Suppose first that there exists an index $i_0\in\{1,\ldots,s\}$ with $v_{e_{i_0}}(\bsa^{(i_0)})=m$. In this case, we distinguish two subcases.
\begin{itemize}
 \item If for some $i_1\in\{1,\ldots,s\}\setminus\{i_0\}$, we have $v_{e_{i_1}}(\bsa^{(i_1)})>0$, then
 $$ V_{m,\e}(\bsA)\ge  v_{e_{i_0}}(\bsa^{(i_0)})+v_{e_{i_1}}(\bsa^{(i_1)})\ge m+1\ge d+1,$$
as $d$ cannot exceed $m$ by definition. In this case, the results in~(a) and~(b) are shown.

 \item Suppose now that there is no index $i_1\in\{1,\ldots,s\}\setminus\{i_0\}$ for which $v_{e_{i_1}}(\bsa^{(i_1)})>0$. In this case, the result 
 in~(a) follows immediately as $V_{m,\e}(\bsA)= v_{e_{i_0}}(\bsa^{(i_0)})=m\ge d$. To show the result in~(b), note that equation~\eqref{eqlindependent} is equivalent to 
 $$\sum_{j=1}^m a_j^{(i_0)}\bsc_j^{(i_0)}=\bszero\in\Field_q^m,$$
where not all $a_j^{(i_0)}$ are zero. However, this is a contradiction to the assumption that $C_{i_0}$ is nonsingular, 
so this situation cannot occur under the hypotheses in~(b).
\end{itemize}
In summary, we have shown the results in~(a) and~(b) for the case where there exists an index $i_0\in\{1,\ldots,s\}$ with $v_{e_{i_0}}(\bsa^{(i_0)})=m$.

Suppose now that $v_{e_i}(\bsa^{(i)})< m$ for all $i\in\{1,\ldots,s\}$. Note that this implies that $e_i| v_{e_i}(\bsa^{(i)})$ for all $i\in\{1,\ldots,s\}$. 
In this case,~\eqref{eqlindependent} boils down to
$$\sum_{i=1}^s \sum_{j=1}^{v_{e_i}(\bsa^{(i)})} a_j^{(i)} \bsc_j^{(i)}=\bszero\in\Field_q^m.$$
As $e_i | v_{e_i}(\bsa^{(i)})$ for all $i\in\{1,\ldots,s\}$, and in view of Definition~\ref{defsystem}, this linear dependence relation can hold only if 
$$V_{m,\e}(\bsA)=\sum_{i=1}^s v_{e_i} (\bsa^{(i)})\ge d+1,$$
and so we obtain the desired results in~(a) and~(b). 

To show the assertion in~(c), assume now that $\delta_{m,\e}(\cC^\perp)\ge d+1$. We have to show that any system
$\{\bsc_j^{(i)}\in\Field_q^m: 1\le j\le d_i, 1\le i\le s\}$ with $0\le d_i\le m$, $e_i|d_i$, $1\le i\le s$, and $\sum_{i=1}^s d_i \le d$ is 
linearly independent over $\Field_q$. Suppose, on the contrary, that such a system were linearly dependent over $\Field_q$, i.e., 
that there exist coefficients $a_j^{(i)}\in\Field_q$, not all 0, such that 
$$\sum_{i=1}^s \sum_{j=1}^{d_i} a_j^{(i)} \bsc_j^{(i)}=\bszero\in\Field_q^m.$$
Define $a_j^{(i)}=0$ for $d_i < j\le m$, $1\le i\le s$, then 
$$\sum_{i=1}^s \sum_{j=1}^m a_j^{(i)} \bsc_j^{(i)} =\bszero\in\Field_q^m.$$
This implies that $\bsA\in\cC^\perp$, and so $V_{m,\e}(\bsA)\ge d+1$. On the other hand, 
$$v_{e_i}(\bsa^{(i)})\le e_i \left\lceil \frac{d_i}{e_i}\right\rceil=d_i\ \mbox{for}\  1\le i\le s,$$
and so 
$$V_{m,\e} (\bsA)=\sum_{i=1}^s v_{e_i}(\bsa^{(i)})\le \sum_{i=1}^s d_i \le d,$$
which is a contradiction. Hence the proof of~(c) is complete.
\end{proof}

We now have the following consequence.

\begin{theorem} \label{thmnetdual}
Let $q$ be a prime power, let $s \ge 1$, $m \ge 1$, and $0\le u\le m$ be integers, and let $\e\in\NN^s$. Let $C_1,\ldots,C_s$ be $m\times m$ matrices 
over $\Field_q$ and  put
$$C:=(C_1^\top | C_2^\top |\cdots | C_s^\top) \in \F_q^{m \times sm}.$$
Let $\cC^{\perp}$ be the dual space of the row space $\cC$ of $C$.
\begin{itemize}
 \item[(a)] If $C_1,\ldots,C_s$ are all nonsingular, then they generate a digital $(u,m,\e,s)$-net over $\Field_q$ if and only if 
 $\delta_{m,\e} (\cC^\perp)\ge m-u+1$.
 \item[(b)] If $C_1,\ldots,C_s$ are arbitrary, then the following assertions hold:
 \begin{itemize}
 \item If $\delta_{m,\e}(\cC^\perp)\ge m-u+1$, this implies that $C_1,\ldots,C_s$ generate a digital $(u,m,\e,s)$-net over $\Field_q$.
 \item If $C_1,\ldots,C_s$ generate a digital $(u,m,\e,s)$-net over $\Field_q$, it follows that $\delta_{m,\e}(\cC^\perp)\ge m-u$.
 \end{itemize}
\end{itemize}
\end{theorem}

\begin{proof}
 The result follows by combining Lemma~\ref{lemnetsystem} and Theorem~\ref{thmsystemdual}.
\end{proof}

For the special case where $\e={\bf 1}=(1,\ldots,1)\in\NN^s$, there exists a stronger version of Theorem~\ref{thmnetdual}, which is the following Theorem~2 
in~\cite{NP01}.

\begin{theorem} \label{thmnetdualNP}
Let $q$ be a prime power, let $s \ge 1$, $m \ge 1$, and $0\le t\le m$ be integers. Let $C_1,\ldots,C_s$ be $m\times m$ matrices over $\Field_q$ and put
 $$C:=(C_1^\top | C_2^\top | \cdots | C_s^\top) \in \F_q^{m \times sm}.$$
Let $\cC^{\perp}$ be the dual space of the row space $\cC$ of $C$.
Then $C_1,\ldots,C_s$ generate a digital $(t,m,s)$-net over $\Field_q$ if and only if $\delta_m (\cC^\perp)\ge m-t+1$, where 
 $\delta_m=\delta_{m,\bsone}$, as introduced in~\cite{NP01}.
\end{theorem}

We now establish a result similar to Lemma~1 in~\cite{NX02}.

\begin{proposition} \label{prNdual}
Let $q$ be a prime power, let $s \ge 2$ and $m \ge 1$ be integers, and let $\e\in\NN^s$. Then from any 
$\Field_q$-linear subspace $\cN$ of $\Field_q^{sm}$ with $\dim (\cN)\ge sm-m$ we obtain a digital $(u,m,\e,s)$-net over 
$\Field_q$ with $u=\max\{0,m-\delta_{m,\e} (\cN)+1\}$. 
\end{proposition}

\begin{proof}
 For $\cC:=\cN^\perp$ we have $\dim (\cC)\le m$, and so $\cC$ is the row space of a suitable digital net over $\Field_q$. 
Now let $u=\max\{0,m-\delta_{m,\e} (\cN)+1\}$. Then $\delta_{m,\e}(\cC^\perp)=\delta_{m,\e}(\cN) \ge m-u+1$, and so, by Theorem~\ref{thmnetdual}, 
$\cC=\cN^\perp$ generates a digital $(u,m,\e,s)$-net over $\Field_q$. 
\end{proof}

\begin{remark} \label{redu} {\rm
We obviously have $v_e(\bsa) \le v_1(\bsa) +e-1$ for any $\bsa \in \F_q^m$, and so
$$
V_{m,\e}(\bsA) \le V_{m,{\bf 1}}(\bsA)+\sum_{i=1}^s (e_i-1)=V_m(\bsA)+\sum_{i=1}^s (e_i-1)
$$
for any $\bsA \in \F_q^{sm}$ (compare with Remark~\ref{revm}). It follows that
$$
\delta_{m,\e}(\cN) \le \delta_{m,{\bf 1}}(\cN)+\sum_{i=1}^s (e_i-1) = \delta_m(\cN)+\sum_{i=1}^s (e_i-1)
$$
for any linear subspace $\cN$ of $\F_q^{sm}$. This is in good accordance with \cite[Proposition~1]{HN13} where the term $\sum_{i=1}^s (e_i-1)$ governs
the transition from $(u,m,\e,s)$-nets to $(t,m,s)$-nets.}
\end{remark}

\section{Applications of the duality theory} \label{secad}

In the following, we present an application of the duality theory in Section~\ref{secduality} and of the theory of global function fields. We use the same notation
and terminology for global function fields as in the monograph~\cite{NX09}. In particular, for any divisor $D$ of a global function field $F$, let
$$
\cL(D) =\{f \in F^*: {\rm div}(f)+D \ge 0\} \cup \{0\}
$$
be the Riemann-Roch space associated with $D$, where ${\rm div}(f)$ denotes the principal divisor of $f \in F^* := F \setminus \{0\}$. Note that $\cL(D)$ is a
finite-dimensional vector space over the full constant field of $F$ (compare with \cite[Section~3.4]{NX09}).

\begin{theorem} \label{thgf}
Let $F$ be a global function field with full constant field $\F_q$ and genus $g$. For an integer $s \ge 2$, let $P_1,\ldots,P_s$ be $s$ distinct places of $F$.
Put $e_i=\deg(P_i)$ for $1 \le i \le s$ and $\e =(e_1,\ldots,e_s) \in \N^s$. Then for every integer $m \ge \max (1,g)$ which is a multiple of ${\rm lcm} (e_1,\ldots,e_s)$,
we can construct a digital $(g,m,\e,s)$-net over $\F_q$.
\end{theorem}

\begin{proof}
We fix integers $s$ and $m$ satisfying the hypotheses of the theorem. It follows from \cite[Corollary~5.2.10(c)]{St} that for a sufficiently large integer $d >
\max_{1 \le i \le s} e_i$, there exist places $Q_1$ and $Q_2$ of $F$ with $\deg(Q_1)=d+1$ and $\deg(Q_2)=d$. We define the divisor $G$ of $F$ by
$$
G=(sm-m+g-1)(Q_1-Q_2).
$$
Then
\begin{equation} \label{eqdg}
\deg(G)=sm-m+g-1.
\end{equation}
(More generally, we can take any divisor $G$ of $F$ such that $\deg(G)$ satisfies~\eqref{eqdg} and the support of $G$ is disjoint from the set $\{P_1,\ldots,P_s\}$
of places.) Take any $f \in \cL(G)$. Fix $i \in \{1,\ldots,s\}$ for the moment. Let $\nu_{P_i}$ be the normalized discrete valuation of $F$ corresponding to the place
$P_i$. We have $\nu_{P_i}(f) \ge 0$ by the choices of $f$ and $G$, and so the local expansion of $f$ at $P_i$ has the form
$$
f=\sum_{j=0}^{\infty} a_j^{(i)}(f)z_i^j
$$
with all $a_j^{(i)}(f) \in \F_{q^{e_i}}$, where $z_i$ is a local parameter at $P_i$. Choose an $\F_q$-linear isomorphism $\psi_i: \F_{q^{e_i}} \to \F_q^{e_i}$ and for
convenience put $k_i=m/e_i \in \N$. Then we define the $\F_q$-linear map $\theta_i: \cL(G) \to \F_q^m$ by
\begin{equation} \label{eqti}
\theta_i(f)=\left(\psi_i(a_{k_i-1}^{(i)}(f)),\psi_i(a_{k_i-2}^{(i)}(f)),\ldots,\psi_i(a_0^{(i)}(f)) \right) \qquad \mbox{for all } f \in \cL(G).
\end{equation}
Furthermore, we define the $\F_q$-linear map $\theta : \cL(G) \to \F_q^{sm}$ by
$$
\theta (f)=(\theta_1(f),\ldots,\theta_s(f)) \qquad \mbox{for all } f \in \cL(G).
$$
By definition, let the $\F_q$-linear space $\cN \subseteq \F_q^{sm}$ be the image of $\theta$.

Now we take any nonzero $f \in \cL(G)$. We put
$$
w_i(f)=\min (k_i,\nu_{P_i}(f)) \qquad \mbox{for } 1 \le i \le s.
$$
We claim that for the weights $v_{e_i}(\theta_i(f))$ we have
\begin{equation} \label{eqth}
v_{e_i}(\theta_i(f))=m-e_iw_i(f) \qquad \mbox{for } 1 \le i \le s.
\end{equation}
If $\nu_{P_i}(f) \ge k_i$, then $a_j^{(i)}(f)=0$ for $0 \le j \le k_i-1$, and so $\theta_i(f)=\bszero \in \F_q^m$. Then $v_{e_i}(\theta_i(f))=0$, and so~\eqref{eqth} holds
in this case. In the remaining case, we have $0 \le h_i := \nu_{P_i}(f) < k_i$. Then $a_j^{(i)}(f)=0$ for $0 \le j < h_i$ and $a_{h_i}^{(i)}(f) \ne 0$. It follows that
$\psi_i(a_j^{(i)}(f))=\bszero \in \F_q^{e_i}$ for $0 \le j < h_i$ and $\psi_i(a_{h_i}^{(i)}(f)) \ne \bszero$. The definition of $\theta_i(f)$ in~\eqref{eqti} shows then that
$$
v_{e_i}(\theta_i(f))=(k_i-h_i)e_i=m-e_i \nu_{P_i}(f),
$$
and so~\eqref{eqth} holds again. Consequently, we obtain
$$
V_{m,\e}(\theta(f))=\sum_{i=1}^s v_{e_i}(\theta_i(f))=sm- \sum_{i=1}^s e_iw_i(f).
$$
We have $\nu_{P_i}(f) \ge w_i(f)$ for $1 \le i \le s$, and so $f \in \cL \left(G-\sum_{i=1}^s w_i(f)P_i \right)$. Since $f \ne 0$, this means that
$$
0 \le {\rm div}(f)+G-\sum_{i=1}^s w_i(f)P_i.
$$
By applying the degree map for divisors and noting that $\deg({\rm div}(f))=0$ according to \cite[Corollary~3.4.3]{NX09}, we deduce that 
$$
0 \le \deg(G)- \sum_{i=1}^s e_iw_i(f)=\deg(G)+V_{m,\e}(\theta(f))-sm.
$$
Therefore by~\eqref{eqdg},
$$
V_{m,\e}(\theta(f)) \ge sm-\deg(G)=m-g+1 > 0.
$$
This shows, in particular, that the map $\theta$ is injective, and also
$$
\delta_{m,\e}(\cN) \ge m-g+1.
$$
Moreover,
$$
\dim(\cN)=\dim(\cL(G)) \ge \deg(G)+1-g=sm-m,
$$
where we applied the Riemann-Roch theorem (see \cite[Theorem~3.6.14]{NX09}) in the second step. The rest follows from Proposition~\ref{prNdual}.
\end{proof}

\begin{remark} \label{regf} {\rm
If we combine Theorem~\ref{thgf} with Proposition~\ref{prprd}, then we usually get many more integers $m \ge \max (1,g)$ for which we can construct a digital
$(g,m,\e,s)$-net over $\F_q$. For instance, if at least one $e_i=1$, then we obtain a digital $(g,m,\e,s)$-net over $\F_q$ for any integer $m \ge \max (1,g)$. }
\end{remark}

\begin{example} \label{exgf} {\rm
Let $q$ be an arbitrary prime power and let $s=q+2$. Let $F$ be the rational function field over $\F_q$. Then $F$ has genus $g=0$ and exactly $q+1$ places of
degree $1$ (the infinite place corresponding to the degree valuation and the $q$ places corresponding to the distinct monic linear polynomials over $\F_q$).
Choose distinct places $P_1,\ldots,P_s$ of $F$ such that $\deg(P_i)=1$ for $1 \le i \le s-1=q+1$ and $\deg(P_s)=2$. Note that $P_s$ corresponds to a monic
irreducible quadratic polynomial over $\F_q$. Then Theorem~\ref{thgf} shows that for every even integer $m \ge 2$ we can obtain a digital $(0,m,\e,q+2)$-net
over $\F_q$ with $\e =(1,\ldots,1,2) \in \N^{q+2}$. In combination with Proposition~\ref{prprd}, we get a digital $(0,m,\e,q+2)$-net over $\F_q$ for any integer
$m \ge 1$. Note that for ${\bf 1}=(1,\ldots,1) \in \N^{q+2}$ and $m \ge 2$, there cannot exist a $(0,m,{\bf 1},q+2)$-net in base $q$, that is, a $(0,m,q+2)$-net
in base $q$, as this would violate a combinatorial bound for nets in~\cite{N87} (see also \cite[Corollary~4.19]{DP} and \cite[Corollary~4.21]{N92}). }
\end{example}

\begin{example} \label{exco} {\rm
Let $q=2$, $s=4$, $m=3$, and $\e =(1,1,1,2)$. The following is a concrete example of a digital $(0,3,\e,4)$-net over $\F_2$. The four generating matrices over
$\F_2$ are given by
$$
C_1=\left(
\begin{array}{ccc}
1 & 0 & 0\\
0 & 1 & 0\\
0 & 0 & 1
\end{array}
\right), \ \
C_2=\left(
\begin{array}{ccc}
0 & 0 & 1\\
1 & 1 & 0\\
0 & 1 & 0
\end{array}
\right), \ \
C_3=\left(
\begin{array}{ccc}
0 & 1 & 1\\
1 & 0 & 1\\
0 & 0 & 1\\
\end{array}
\right), \ \
C_4=\left(
\begin{array}{ccc}
1 & 0 & 1\\
0 & 1 & 0\\
0 & 0 & 0
\end{array}
\right).
$$
It is easily verified that the row vectors of these matrices form a $(3,3,\e,4)$-system over $\F_2$ (note that from $C_4$ we take either no row vectors or
the first two row vectors). Hence it follows from Lemma~\ref{lemnetsystem} that $C_1,C_2,C_3,C_4$ do indeed generate a digital $(0,3,\e,4)$-net over $\F_2$. }
\end{example}

\end{document}